\documentclass{article}
\usepackage{amssymb}
\title{{\bf  Group Partitions via Commutativity and Related Topics}}
\author{ {\bf A. Mahmoudifar}, {\bf A. R. Moghaddamfar} and  {\bf F. Salehzadeh}}

\newenvironment{proof}{\noindent {\em {Proof}}.}{$\square$
\medskip}
\newtheorem{theorem}{Theorem}[section]
\newtheorem{definition}[theorem]{Definition}
\newtheorem{corollary}[theorem]{Corollary}

\newtheorem{lm}[theorem]{Lemma}

\begin{document}
\newcommand{\f}{\frac}
\newcommand{\sta}{\stackrel}
\maketitle
\begin{abstract}
\noindent  
Let $G$ be a nonabelian group, $A\subseteq G$ an abelian subgroup and $n\geqslant 2$ an integer.
We say that $G$ has an $n$-abelian partition with respect to $A$, 
if there exists a partition of $G$ into  $A$ and $n$ disjoint commuting subsets $A_1,  A_2,  \ldots, A_n$ of $G$,  such that $|A_i|>1$ for each $i=1, 2, \ldots, n$. 
We first classify all nonabelian groups, up to isomorphism,  which have an $n$-abelian partition for $n=2, 3$. 
Then, we  provide
some formulas concerning the number of spanning trees of commuting graphs
associated with certain finite groups.  Finally, we point out some ways 
to finding  the number of spanning trees of
the commuting graphs of some specific groups.
  \end{abstract}


\renewcommand{\baselinestretch}{1}
\def\thefootnote{ \ }
\footnotetext{{\em $2010$ Mathematics Subject Classification}: 05C25,  20D06.}
\section{Introduction and Motivation}
Let $\Gamma$ be a simple graph and  $m, n$ two non-negative integers. We say that $\Gamma$ is {\em$(m,n)$-partitionable}
 if its vertex set can be partitioned into $m$ independent sets $I_1, \ldots, I_m$ and $n$ cliques $C_1, \ldots, C_n$; that is 
 $$V_\Gamma=I_1\uplus I_2\uplus\cdots\uplus I_m \uplus C_1\uplus C_2\uplus\cdots\uplus C_n.$$
Such a partition of $V_\Gamma$ is called a {\em$(m, n)$-partition} of $\Gamma$ (see \cite{Brandstadt}). 
We shall note some special cases: $(1,1)$-partitionable graphs are called split graphs (see \cite{Foldes-Hammer}), $(1,0)$-partitionable graphs are called edgeless graphs,
 $(0,1)$-partitionable graphs are called complete graphs.
In particular,  in the case when $m=0$ or $n=0$,  we essentially split $\Gamma$ into  $n$ cliques, $V_\Gamma=C_1\uplus C_2\uplus\cdots\uplus C_n$, or $m$ independent sets,
 $V_\Gamma=I_1\uplus I_2\uplus\cdots\uplus I_m$, respectively. 
 
 We now focus our attention on a graph associated with a finite group-  the so-called commuting graph. Let $G$ be a finite group and $X$ a nonempty subset of $G$.
The {\em commuting graph} ${\cal C}(X)$,  has
$X$ as its vertex set with two distinct elements of
$X$ joined by an edge when they {\em commute} in $G$.
Commuting graphs have been investigated by many authors in various contexts, see for instance \cite{Britnell, Das, Mahmoudifar}. Clearly, ${\cal C}(G)$ is $(0,n)$-partitionable if and only if $G$ 
can be partitioned into $n$ commuting subsets. This suggests the following definition.
\begin{definition}\label{def1}{\rm  Let $G$ be a nonabelian group, $A\subseteq G$ an abelian subgroup and $n\geqslant 2$ an integer.
We say that $G$ has an {\em $n$-abelian partition with respect to $A$}, 
if there exists a partition of $G$ into  $A$ and $n$ disjoint commuting subsets $A_1,  A_2,  \ldots, A_n$ of $G$,
$G=A\uplus A_1\uplus A_2\uplus \cdots \uplus A_n
$, such that $|A_i|>1$ for each $i=1, 2, \ldots, n$. }
\end{definition}

Note that, the condition $n\geqslant 2$ in Definition \ref{def1} is needed, because if $n=1$, we have  $G=A\uplus A_1$, and so $G=\langle A_1\rangle$.
Since $A_1$ is a commuting set, this would imply $G$ is abelian, which is not the case.  On the other hand, 
the structure of groups $G$ which have an $n$-abelian partition for $n=2, 3$, is obtained (Theorems \ref{th-1} and \ref{th-2}).

\noindent {\em Remark.} Let $Z=Z(G)$.  If $|Z|\geqslant 2$, $|G:Z|=1+n\geqslant 4$ and $T=\{x_0, x_1, \ldots, x_n\}$ is 
a transversal for $Z$ in $G$, where $x_0\in Z$, then,  as the cosets of the centre are 
abelian subsets of $G$, we have the following $n$-abelian partition for $G$ with respect to $Z$:
$$G=Z\uplus Zx_1\uplus Zx_2\uplus \cdots \uplus Zx_n.$$ 
However, there are centerless groups $G$  for which there is no $n$-abelian partition with respect to an abelian subgroup,
for every $n$. For example, consider the symmetric group ${\Bbb S}_3$ on $3$ letters.

Clearly, when $1\in X$, ${\cal C}(X)$ is connected, so one can talk about the number of spanning trees  (or {\em tree-number})
of this graph, which is denoted by $\kappa(X)$.  Moreover, it is easy to see that $X$ is a  
 commuting subset of $G$  (i.e., all 
pairs of its elements commute) if and only if ${\cal C}(X)$ is a complete graph. Thus,  if $X$ is a commuting subset of $G$, then 
by Cayley's formula we obtain  $\kappa(X)=|X|^{|X|-2}$. 
In \cite{Mahmoudifar}, in particular, it was proved that, for every
 Frobenius group $G$ with core $F$ and complement $C$,  $\kappa(G)=\kappa(F)\kappa(C)^{|F|}$.
 Moreover, it was shown in \cite{Mahmoudifar} that the alternating group ${\Bbb A}_5$ can be characterized by $\kappa ({\Bbb A}_5)$  in the class of nonsolvable groups.
Here, we are interested in the problem of finding 
the $\kappa(G)$ for some specific groups $G$ (Theorems  \ref{char2}-\ref{pp-abelian}, Corollaries \ref{2-abelian-trees}, \ref{3-abelian-trees}, and Table 1). 

The outline of the paper is as follows. 
In Section 2, we provide a number of basic results related to the tree-numbers.
In Section 3, the structure of groups $G$ which have a $n$-abelian partition for $n\in \{2, 3\}$, is obtained.   
Finally, in Section 4,  some explicit formulas for the tree-numbers of commuting graphs associated with certain groups are obtained.  

All notation and terminology for groups and graphs are standard. In addition, 
the  spectrum $\omega(G)$ of a finite group $G$ is the set of its element orders. 
It is closed under divisibility relation and so determined uniquely through the set $\mu (G)$ 
of those elements in  $\omega(G)$  that are  maximal under the divisibility relation. 
Following S. M. Belcastro and G. J. Sherman \cite{Belcastro} we  denote 
by $\#{\rm Cent}(G)$ the number of distinct centralizers in $G$. We 
shall say that a group $G$
is $n$-centralizer if $\#{\rm Cent}(G)= n$. 
\section{Auxiliary Results}
Let $G$ be a nonabelian group and $A\subset G$ an abelian subgroup. 
If $G$ has an $n$-abelian partition  with respect to  $A$, 
$G=A\uplus A_1\uplus \cdots \uplus A_n$, then, we have 
$${\cal C}(A)=K_{|A|}, \ \ \  {\cal C}(A_i\cup \{1\})=K_{|A_i|+1}, \ \ i=1, 2, \ldots, n.$$
Therefore, by Corollary  2.7 in \cite{Mahmoudifar}, we get:
$$\kappa(G)\geqslant  |A|^{|A|-2} \prod_{i=1}^{n} (|A_i|+1)^{|A_i|-1}.$$
In particular, if the center $Z(G)$ of the group $G$ is nontrivial, and 
$$G=Z(G)\uplus Z(G)x_1\uplus Z(G)x_2\uplus \cdots \uplus Z(G)x_n,$$ 
is an $n$-abelian partition for $G$ (a coset decomposition of $G$), then 
$$\kappa(G)\geqslant  |Z(G)|^{|Z(G)|-2} (|Z(G)|+1)^{(|Z(G)|-1)n}.$$

A {\em noncommuting set} of a group $G$ (i.e., an independent set in commuting graph ${\cal C}(G)$) has the property that no two of its elements commute under the group operation.  We denote by ${\rm nc}(G)$ 
 the maximum cardinality of any noncommuting  set of $G$ (the independence number of  ${\cal C}(G)$).  Denote by ${\rm k}(G)$ the number of distinct conjugacy classes of $G$. If $G$ has an $n$-abelian partition, then  
 the pigeon-hole principle gives ${\rm nc}(G)\leqslant n+1$. Thus, by  Corollary 2.2 (a) in \cite{Bertram}, we obtain $$|G|\leqslant {\rm nc}(G)\cdot {\rm k}(G)\leqslant (n+1){\rm k}(G),$$
which immediately implies that
\begin{equation}\label{eq-4}
 n\geqslant \left\lfloor \frac{|G|}{{\rm k}(G)}\right\rfloor-1.
\end{equation}
 Therefore, we have found a lower bound for $n$, when ${\rm k}(G)$ is known.
 
{\em Some Examples.}  Let $G=L_2(q)$, where $q\geqslant 4$ is a power of $2$.  We know that $|G|=q(q^2-1)$ and ${\rm k}(L_2(q))=q+1$. Thus, if $G$ has an $n$-abelian partition, then by Eq. (\ref{eq-4}), we get  
 $n\geqslant  q^2-q-1$.  In particular, since ${\Bbb A}_5\cong L_2(4)$,  if ${\Bbb A}_5$ has an $n$-abelian partition, then
 $n\geqslant 11$.  In fact, ${\Bbb A}_5$ has  a $20$-abelian partition, as follows: 
 $${\Bbb A}_5=A\uplus A_1^\#\uplus A_2^\#\uplus \cdots \uplus A_{20}^\#,$$
 where   $A_i^\#=A_i\setminus \{1\}$, for every $i$, and 
 \begin{itemize}
 \item[]  $A, A_1, \ldots, A_5$ are Sylow $5$-subgroups of order $5$, 
 \item[]  $A_6, A_7, \ldots, A_{15}$ are Sylow $3$-subgroups of order $3$,  
 \item[]  $A_{16}, A_{17}, \ldots, A_{20}$ are Sylow $2$-subgroups of order $4$. 
 \end{itemize}
 Similarly,  if $G_1={\rm GL}(2,q)$ and $G_2={\rm GL}(3,q)$, $q$  a prime power, then we have   \begin{itemize}
 \item[]    $|G_1|=(q^2-q)(q^2-1)$ and  ${\rm k}(G_1)=q^2-1$,  while
 \item[]  $|G_2|=(q^3-1)(q^3-q)(q^3-q^2)$ and ${\rm k}(G_2)=q^3-q$.
 \end{itemize}
 Again, if $G_i$ has an $n_i$-abelian partition, for $i=1, 2$, by Eq. (\ref{eq-4}), we obtain 
 $n_1\geqslant  q(q-1)-1$ and $n_2\geqslant q^2(q^3-1)(q-1)-1$. 
 
In the sequel, we establish some notation which will be used repeatedly.
Given a graph $\Gamma$, we denote by $A_\Gamma$ and $D_\Gamma$ the adjacency matrix and the diagonal matrix of vertex degrees of $\Gamma$, respectively. The {\em Laplacian matrix} of $\Gamma$ is defined
as $L_\Gamma=D_\Gamma-A_\Gamma$. Clearly, $L_\Gamma$ is a real symmetric matrix and its eigenvalues are nonnegative real numbers.
The Laplacian spectrum of $\Gamma$ is
$${\rm Spec}(L_\Gamma)=\left(\mu_1(\Gamma), \mu_2(\Gamma), \ldots, \mu_n(\Gamma)\right),$$
where $\mu_1(\Gamma)\geqslant \mu_2(\Gamma)\geqslant \cdots\geqslant \mu_n(\Gamma)$,
are the eigenvalues of $L_\Gamma$ arranged in weakly decreasing order, and $n=|V\Gamma|$.
Note that,  $\mu_n(\Gamma)$ is always $0$, because each row sum of $L_\Gamma$ is 0.
Instead of $A_\Gamma$, $D_\Gamma$, $L_\Gamma$, and $\mu_i(\Gamma)$  we simply write $A$, $D$, $L$,  and $\mu_i$  if it does not lead to confusion.

For a graph with $n$ vertices and Laplacian spectrum $\mu_1\geqslant \mu_2\geqslant \cdots\geqslant \mu_n$
it has been proved \cite[Corollary 6.5]{Biggs} that:
\begin{equation}\label{eq1}
\kappa(\Gamma)=\frac{\mu_1 \mu_2 \cdots \mu_{n-1}}{n}.
\end{equation}

The vertex-disjoint union of the graphs $\Gamma_1$ and $\Gamma_2$ is denoted by $\Gamma_1\oplus\Gamma_2$.  We shall write $m\Gamma$ instead of $\Gamma \oplus \Gamma \oplus \cdots \oplus \Gamma$ ($m$ times).
Define the {\em join} of $\Gamma_1$ and $\Gamma_2$ to be $\Gamma_1 \vee \Gamma_2=(\Gamma_1^c\oplus\Gamma_2^c)^c$,
where $\Gamma^c$ signifies the complement of $\Gamma$. Actually, this is the vertex-disjoint union of the two graphs,
and adding edges joining every vertex of  $\Gamma_1$  to every
vertex of $\Gamma_2$.
Now, one may easily prove (see also \cite{Merris}):
\begin{lm}\label{elementary0}
Let $\Gamma_1$ and $\Gamma_2$ be two graphs on disjoint sets  with $m$ and $n$ vertices, respectively. If
$${\rm Spec}(L_{\Gamma_1})=\left(\mu_1(\Gamma_1), \mu_2(\Gamma_1), \ldots, \mu_m(\Gamma_1)\right),$$
and
$${\rm Spec}(L_{\Gamma_2})=\left(\mu_1(\Gamma_2), \mu_2(\Gamma_2), \ldots, \mu_n(\Gamma_2)\right),$$ then, there hold:
\begin{itemize}
\item[{\rm (1)}] the eigenvalues of Laplacian matrix  $L_{\Gamma_1\oplus \Gamma_2}$ are:
$$\mu_1(\Gamma_1), \  \ldots, \ \mu_{m}(\Gamma_1),  \ \mu_1(\Gamma_2), \  \ldots, \ \mu_{n}(\Gamma_2).$$

\item[{\rm (2)}] the eigenvalues of Laplacian matrix  $L_{\Gamma_1\vee \Gamma_2}$ are:
$$m+n, \ \mu_1(\Gamma_1)+n, \  \ldots, \  \mu_{m-1}(\Gamma_1)+n, \ \mu_1(\Gamma_2)+m, \ \ldots, \ \mu_{n-1}(\Gamma_2)+m, \ 0.$$
\end{itemize}
\end{lm}
\begin{lm}\label{join-integer}  Let $\Gamma$ be any graph on $n$ vertices with Laplacian spectrum
$\mu_1 \geqslant \mu_2 \geqslant  \cdots \geqslant  \mu_n$.
If $m$ is an integer, then the following product
$$(\mu_1+m)(\mu_2+m)\cdots (\mu_{n-1}+m),$$
is also an integer which is  divisible by $m$.
\end{lm}
\begin{proof}  Consider the characteristic polynomial of the Laplacian matrix $L=L_\Gamma$:
$$\sigma(\Gamma; \mu)=\det (\mu I-L)=\mu^n+c_1\mu^{n-1}+\cdots+c_{n-1}\mu+c_n.$$
First, we observe that the coefficients $c_i$ are integers \cite[Theorem 7.5]{Biggs}, and in particular, $c_n=0$.
This forces  $\sigma(\Gamma; -m)$ is an integer, which is  divisible by $m$.
Moreover, we have
$$\sigma(\Gamma; \mu)=(\mu-\mu_1)(\mu-\mu_2)\cdots (\mu-\mu_n),$$
and since $\mu_n=0$, we obtain
$$\sigma(\Gamma; -m)=(-1)^n m(\mu_1+m)(\mu_2+m)\cdots (\mu_{n-1}+m).$$
 The result now follows.
\end{proof}

A {\em universal} vertex is a vertex of a graph that is adjacent to all other vertices of the graph.
\begin{lm}\label{full}
Let  a graph $\Gamma$ with $n$ vertices contain $m<n$ universal vertices. Then $\kappa(\Gamma)$ is divisible
by $n^{m-1}$.
\end{lm}
\begin{proof}  Let $W$ be the set of universal vertices. Clearly, $\Gamma=K_m\vee (\Gamma -W)$.
Since the Laplacian for the complete graph $K_m$ has eigenvalue $0$ with multiplicity $1$ and eigenvalue $m$ with multiplicity $m-1$,  it follows by Lemma \ref{elementary0}  that $L_\Gamma$ has eigenvalue
$n$  with multiplicity at least $m$. The result is now immediate from Lemma \ref{join-integer} and Eq. (\ref{eq1}).
\end{proof}

It is easy to see that for a group $G$, its center consists of the universal vertices of ${\cal C}(G)$.
Therefore, if $Z(G)$ is of order $m$, then we have
\begin{equation}\label{eq2}
{\cal C}(G)=K_m\vee \Delta(G),
\end{equation}
where $\Delta(G)={\cal C}(G\setminus Z(G))$.
The following corollary is now immediate from Lemma \ref{full}.
\begin{corollary} Let $G$ be a nonabelian group of order $n$ with the center of order $m$. Then $\kappa(G)$ is divisible by $n^{m-1}$.
\end{corollary}

We add further information about the commuting graph ${\cal C}(G)$. Let $G$ be a {\em nonabelian} group with $|G|=n$ and $|Z(G)|=m$.
Then $\Delta=\Delta(G)$ is a graph on $\nu=n-m>1$ vertices. For the characteristic polynomial of the Laplacian matrix $L_{\Delta}$, we shall write
$$\sigma(\Delta; \mu)=\det(\mu I-L_{\Delta})=\mu^{\nu}+c_1\mu^{\nu-1}+\cdots+c_{\nu-1}\mu. \ \ \ \ (\mbox{note that} \ c_\nu=0.)$$
If ${\rm Spec}(L_{\Delta})=\left(\mu_1, \mu_2, \ldots, \mu_{\nu-1}, 0\right)$, then we may write
$$\sigma(\Delta; \mu)=\mu (\mu - \mu_1)(\mu -\mu_2)\cdots (\mu -\mu_{\nu-1}).$$
Moreover, by Lemma \ref{elementary0}, the eigenvalues of Laplacian matrix  $L_{{\cal C}(G)}$ are:
$$n, \underbrace{n, \ n, \ \ldots, \ n}_{m-1}, \  \underbrace{\mu_1+m, \ \mu_2+m, \ \ldots, \ \mu_{\nu-1}+m}_{\nu-1}, \ 0. $$
It follows immediately using Eq. (\ref{eq1}) that
$$\kappa(G)=n^{m-1}(\mu_1+m)(\mu_2+m)\cdots (\mu_{\nu-1}+m),$$
or equivalently
$$\kappa(G)=n^{m-1}(-1)^{\nu}\sigma(\Delta; -m)/m.$$
As $\sigma(\Delta; -m)/m$ is an integer, so $n^{m-1}$ divides $\kappa(G)$. 

In particular, if $G$ is a centerless 
group (i.e., $Z(G)=1$), then we have 
$$\kappa(G)=(-1)^{n-1}\sigma(\Delta; -1)=(\mu_1+1)(\mu_2+1)\cdots (\mu_{n-2}+1).$$

As pointed out in the Introduction, if a group $G$ has the {\em nontrivial}  center  $Z=Z(G)$  and  $[G:Z]=1+n\geqslant 4$, then the coset decomposition $$G =\biguplus_{i=0}^{n} Zx_i,$$ is 
an $n$-abelian partition of $G$.  Let $|G|=n$ and $|Z|=m>1$.  Put $t=n/m$, and   
 \begin{equation}\label{eq3}
{\cal C}_0(G)=K_m\vee \left(t-1\right)K_m.
\end{equation}
Clearly, ${\cal C}_0(G)$ is a subgraph of ${\cal C}(G)$, and so $\kappa(G)\geqslant \kappa({\cal C}_0(G))$.  
\begin{corollary}\label{cor-1}  With the above notation,  we have
$$\kappa(G)\geqslant n^{m-1} m^{n-m-1} 2^{(t-1)(m-1)}.$$
\end{corollary}
\begin{proof} 
By Lemma \ref{elementary0}, the eigenvalues of Laplacian matrix  $L_{{\cal C}_0(G)}$ are:
$$n, \underbrace{n, \ n, \ \ldots, \ n}_{m-1}, \  \underbrace{2m, \ 2m, \ \ldots, \ 2m}_{(t-1)(m-1)},  \  \underbrace{m, \ m, \ \ldots, \ m}_{t-2}, \ 0. $$
Using Eq. (\ref{eq1}) and simple computations, we obtain 
$$\kappa(G)\geqslant \kappa({\cal C}_0(G))=n^{m-1} m^{n-m-1} 2^{(t-1)(m-1)},$$
and the result follows. \end{proof}

 \begin{lm}\label{lm-41} {\rm \cite[Lemma 4.1]{Pyber}} 
Let $\{g_1, \ldots, g_m\}$ be a largest noncommuting subset of $G$.  Then $\cap_{i=1}^{m} C_G(g_i)$ is an abelian subgroup of $G$.
\end{lm}
\begin{proof}  Assume the contrary and choose $a, b\in \cap_{i=1}^{m} C_G(g_i)$ such that $ab\neq ba$. Then it is easy to see that $\{a, bg_1, bg_2, \ldots, bg_m\}$ is a noncommuting subset of $G$, 
a contradiction.  
\end{proof}
\section{Groups having an $n$-abelian partition}
Given a finite group $G$,  we denote by  ${\rm cs}(G)$ the set of conjugacy class sizes of $G$.
 It${\rm \hat{o}}$  proved that  \cite[Theorem 1]{Ito} {\em if ${\rm cs}(G)=\{1, m\}$,
then  $G$ is a direct product of a Sylow $p$-group of $G$ with an abelian group. In particular, then
$m$ is a power of  $p$}.
\begin{theorem}\label{th-1}
The following conditions on a nonabelian group $G$ are equivalent:
\begin{itemize}
\item[$(1)$] $G$ has a $2$-abelian partition with respect to an abelian subgroup $A$.
\item[$(2)$]  $G=P\times Q$, where $P\in {\rm Syl}_2(G)$ with $P/Z(P)\cong \mathbb{Z}_2\times \mathbb{Z}_2$  and $Q$ is abelian,
and $A=\langle Z(G), t\rangle$, where $t$ is an involution outside of $Z(G)$.
\end{itemize}
\end{theorem}
\begin{proof} $(1)\Rightarrow (2)$  Suppose that $G$ is a nonabelian group, which has a  $2$-abelian partition
$G=A\uplus A_1\uplus A_2$.  First of all, we notice that  every noncommuting set of $G$ can have at most three elements.
Now fix a  noncentral element $x$  of $G$.  Since $C_G(x)<G$, we can choose $y\in G$, such that  $x$ and $y$  do not commute.
 It is well known that a group cannot be written as the union of two proper subgroups,
thus  $C_G(x)\cup  C_G(y)<G$, and so  we can choose  $z$  in $G$, such that  $B=\{x, y, z\}$
is a  noncommuting set of $G$. Now, we have
$$G=C_x\cup C_y\cup C_z,$$
where $C_x=C_G(x)$, $C_y=C_G(y)$ and $C_z=C_G(z)$. Put $K=C_x\cap C_y\cap C_z$, which is 
an abelian subgroup of $G$, by Lemma \ref{lm-41}.
Indeed, by a  result of Scorza \cite{Scorza}, we have
\begin{itemize}
\item[{\rm (a)}] $[G:C_x]=[G:C_y]=[G:C_z]=2$,
\item[{\rm (b)}]  $K=C_x\cap C_y=C_x\cap C_z=C_y\cap C_z$, and
\item[{\rm (c)}] $K$ is a normal subgroup of $G$ and  the factor group $G/K$ is isomorphic to the Klein Four Group.
\end{itemize}
 Thus  $|x^G|=2$, and
since $x\in G\setminus Z(G)$
was arbitrary, it follows that  $\mbox{cs}(G)=\{1, 2\}$.
By It${\rm \hat{o}}$'s result \cite[Theorem 1]{Ito}, $G=P\times Q$, where $P\in {\rm Syl}_2(G)$  is a  nonabelian and $Q\leqslant Z(G)$.

On the other hand, $B$ is a maximal noncommuting set of $G$,  which forces $C_t\setminus K$ to be a commuting set of $G$
for each $t\in B$, and so the centralizer $C_t$ is abelian, because $C_t=\langle C_t\setminus K\rangle$. This implies that
$K=Z(G)$, and so $$\frac{P}{Z(P)}\cong \frac{P\times Q}{Z(P)\times Q}=\frac{G}{Z(G)}\cong \mathbb{Z}_2\times \mathbb{Z}_2,$$
and the proof is complete.

$(2)\Rightarrow (1)$  Let $\{t_1, t_2, t_3, t_4\}$ be a transversal for $Z(G)$ in $G$, with $t_1\in Z(G)$.
Then, $G$ is a disjoint union:
 $$G=Z(G)\cup Z(G)t_2\cup Z(G)t_3\cup Z(G)t_4.$$
Put $A=Z(G)\cup Z(G)t_2$, $A_1=Z(G)t_3$ and $A_2=Z(G)t_4$.
Then $A$ is an abelian group (since $t_2^2\in Z(G)$), $A_1$ and $A_2$ are commuting sets, and
$G=A\uplus A_1\uplus A_2$ is a 2-abelian partition of $G$.
\end{proof}
\begin{theorem}\label{th-2}
The following conditions on a nonabelian group $G$ are equivalent:
\begin{itemize}
\item[$(1)$] $G$ has a $3$-abelian partition with respect to an abelian subgroup $A$.
\item[$(2)$]  $|Z(G)|\geqslant 2$ and $G/Z(G)$ is isomorphic to one of the following groups: 
$${\Bbb Z}_2\times {\Bbb Z}_2,  \   {\Bbb Z}_3\times {\Bbb Z}_3, \  {\Bbb S}_3.$$
In the first case, $A=Z(G)$, and in two other cases $A=\langle Z(G), x\rangle$, where $x$ is an element of order $3$ outside of $Z(G)$ .
\end{itemize}
\end{theorem}
\begin{proof} $(1)\Rightarrow (2)$  Suppose that $G$ is a nonabelian group, which has a  $3$-abelian partition
$G=A\uplus A_1\uplus A_2\uplus A_3$.  
First of all, we notice that   ${\rm nc}(G)=3$ or $4$. 
It now follows from Lemma 2.4 in \cite{Abdollahi} that  $G$ is either $3$-centralizer or $4$-centralizer, respectively. Therefore, by Theorems 2 and 4 in 
\cite{Belcastro}, we conclude that $G$ modulo its center is isomorphic to one of the groups: ${\Bbb Z}_2\times {\Bbb Z}_2$, ${\Bbb Z}_3\times {\Bbb Z}_3$, or  ${\Bbb S}_3$, as required.  Finally, since $G$ is a nonabelian group 
with at least $7$ elements, $|Z(G)|\geqslant 2$.

$(2)\Rightarrow (1)$  Let $Z=Z(G)$. We treat separately the different cases:
\begin{itemize}
\item[$\rm (a)$] $G/Z\cong {\Bbb Z}_2\times {\Bbb Z}_2$. In this case, there are noncentral
 elements $x_1$, $x_2$, and $x_3$ of $G$ such that $G=Z\uplus  Zx_1\uplus  Zx_2\uplus  Zx_3$,  which is a $3$-abelian partition of $G$.
\item[$\rm (b)$]  $G/Z(G)\cong {\Bbb Z}_3\times {\Bbb Z}_3$.   In this case, we have $$G/Z\cong \langle Zx, Zy \ | \ x^3, y^3, [x,y]\in Z\rangle,$$ which implies that 
$$G=Z\cup Zx\cup Zx^2\cup Zy\cup Zy^2\cup Zxy\cup Zx^2y^2\cup Zxy^2\cup Zx^2y.$$ 
We put
\begin{itemize}
\item[] $A:=Z\cup Zx\cup Zx^2=\langle Z, x\rangle$,
\item[]  $A_1:=Zy\cup Zy^2=\langle Z, y\rangle \setminus Z$,
\item[] $A_2:=Zxy\cup Zx^2y^2=\langle Z, xy\rangle \setminus Z$,
\item[] $A_3:=Zxy^2\cup Zx^2y=\langle Z, xy^2\rangle \setminus Z$.
\end{itemize}
Then $G=A\uplus A_1\uplus A_2\uplus A_3$  is a $3$-abelian partition of $G$.
\item[$\rm(c)$] $G/Z\cong {\Bbb S}_3$. In this case, we have $G/Z\cong \langle Zx, Zy \ | \ x^3, y^2, (xy)^2\in Z\rangle$, which implies that 
$$G=Z\cup Zx\cup Zx^2\cup Zy\cup Zyx\cup Zyx^2.$$ 
Put $A:=Z\cup Zx\cup Zx^2=\langle Z, x\rangle$. Then, $G=A\uplus  Zy\uplus  Zyx\uplus  Zyx^2$  is a $3$-abelian partition of $G$.
\end{itemize}
The proof is complete. \end{proof}
\section{ Computing some explicit formulas for $\kappa(G)$}
We here consider the problem of finding the tree-number of commuting graphs associated with
certain finite groups. 
Our first result concerns Frobenius groups. 
\begin{theorem}
 Let $H\subset G$ be a subgroup of nonabelian group $G$ such that the commuting graph ${\cal C}(G\setminus H)$
is empty. Then, $H$ is abelian of odd order and $G$ is a Frobenius group with kernel $H$ and complement $\mathbb{Z}_2$. In particular, $\kappa(G)=|H|^{|H|-2}$.
\end{theorem}
\begin{proof} Let $\overline{H}=G\setminus H$. Clearly, $\overline{H}$ is nonempty and  $G=H\uplus \overline{H}$. If $A<G$ is an abelian subgroup,
then either $A\leqslant H$ or $|A|=2$. In order to prove this, note that
$A\cap \overline{H}$ can have at most one element, so if $A$ is not
contained in $H$, then the subgroup $A\cap H$ has order $|A|-1$.
Since this must divide $|A|$, we conclude that $|A|=2$.

Let $\bar{h}$ be an arbitrary element in $\overline{H}$. If $1\neq g\in C_G(\bar{h})$, then $\langle g, \bar{h}\rangle$ is abelian and is not
contained in $H$, so it has order $2$, and thus $g=\bar{h}$ and $\bar{h}$ has
order 2 and $C_G(\bar{h})=\langle \bar{h}\rangle$.
Now $\bar{h}$ is contained in some Sylow $2$-subgroup $S$ of $G$. Since
$Z(S)$ is nontrivial, $Z(S)\leqslant C_G(\bar{h})=\langle \bar{h}\rangle$, so
$\bar{h}$ is central in $S$, and thus $S\leqslant C_G(\bar{h})=\langle
\bar{h}\rangle$ has order $2$. Then $|G|=2n$, where $n$ is odd. It
follows that $G$ has a normal subgroup $N$ of order $n$.
Every element of $\overline{H}$ has order $2$ so $N\leqslant H$, and thus
$N=H$, and $H$ has order $n$. Finally, since $\bar{h}$ centralizes no nonidentity element of $H$, the result follows.

For the last statement, we note that $${\cal C}(G)=K_1\vee \left(K_{h-1}\oplus K_h^c\right),$$
where $h=|H|$. We apply Lemma \ref{elementary0} to $\Gamma_1=K_1$ and $\Gamma_2=K_{h-1}\oplus K_h^c$, and use Eq. (\ref{eq1}).
\end{proof}

In what follows, we shall give an explicit formula for $\kappa(L_2(2^n))$.
Let $G=L_2(q)$, where $q=2^n\geqslant 4$.  
Before we start,  we need some well known facts about the simple groups $G$, which are proven in \cite{Hup}:
\begin{itemize}
\item[{\rm (1)}]   $|G|=q(q^2-1)$ and $\mu(G)=\{2, q-1, q+1\}$.

\item[{\rm (2)}] Let $P$ be a Sylow $2$-subgroup of $G$. Then  $P$ is an elementary abelian $2$-group of order
$q$, which is a TI-subgroup, and  $|N_G(P)|=q(q-1)$.

\item[{\rm (3)}] Let $A\subset G$ be a cyclic subgroup of order $q-1$. Then $A$ is a TI-subgroup and
the normalizer $N_G(A)$ is a dihedral group of order $2(q-1)$.

\item[{\rm (4)}] Let $B\subset G$ be a cyclic subgroup of order $q+1$. Then  $B$ is a TI-subgroup
and the normalizer $N_G(B)$ is a dihedral group of order $2(q+1)$.
\end{itemize}

We recall that a subgroup  $H\leqslant G$  is a {\em TI-subgroup} (trivial intersection subgroup) if for every $g\in G$, either $H^g=H$ or $H\cap H^g=\{1\}$.

\begin{theorem} \label{char2} Let  $q=2^n$, where $n\geqslant 2$ is a natural number. Then there holds
$$\kappa(L_2(q))=q^{(q-2)(q+1)} (q-1)^{(q-3)q(q+1)/2}(q+1)^{(q-1)^2q/2}.$$
In particular, when $q=4$, we have $L_2(4)\cong L_2(5)\cong {\Bbb A}_5$ and so $$\kappa(L_2(5))=\kappa({\Bbb A}_5)=2^{20}\cdot 3^{10}\cdot 5^{18}.$$
\end{theorem}
\begin{proof}
Let $G=L_2(q)$, $q=2^n\geqslant 4$.
As already mentioned,  $G$ contains abelian subgroups
$P$, $A$, $B$, of orders $q$, $q-1$, $q+1$, respectively,
every two  distinct conjugates of them intersect trivially and every element
of $G$ is a conjugate of an element in $P\cup A\cup B$.
Let $$G=N_Pu_1\cup \cdots \cup N_Pu_r= N_Av_1\cup \cdots \cup N_Av_s= N_Bw_1\cup \cdots \cup N_Bw_t,$$
be coset decompositions of $G$ by  $N_P=N_G(P)$, $N_A=N_G(A)$ and $N_B=N_G(B)$, where
$r=[G:N_P]=q+1$, $s=[G:N_A]=q(q+1)/2$ and  $t=[G:N_B]=(q-1)q/2$.
Then, we have
$$G=P^{u_1}  \cup \cdots \cup P^{u_r}\cup A^{v_1}\cup \cdots \cup A^{v_s}\cup  B^{w_1} \cup \cdots \cup B^{w_t}.$$
Note that if $x$ is a nonidentity element of $P^{u_i}$ (resp.  $A^{v_j}$, $B^{w_k}$), the centralizer $C_G(x)$ coincides with $P^{u_i}$ (resp.  $A^{v_j}$, $B^{w_k}$) for $1\leqslant i\leqslant r$ (resp. $1\leqslant j\leqslant s$, $1\leqslant k\leqslant t$). This shows that 
$${\cal C}(G)=K_1\vee \left(rK_{q-1}\oplus sK_{q-2}\oplus tK_{q}\right). $$
Applying Cayley's formula to commuting graphs associated with these abelian subgroups yields
$$\kappa(G)= \kappa(P)^r\cdot  \kappa(A)^s\cdot \kappa(B)^t=q^{(q-2)r}\cdot (q-1)^{(q-3)s}\cdot
 (q+1)^{(q-1)t},$$
and the result follows.
\end{proof}

\noindent {\em Remark.}  We know that $L_2(2)\cong S_3\cong {\Bbb Z}_3\rtimes {\Bbb Z}_2$  is a Frobenius group of order $6$.  It is routine to check that  $\kappa(L_2(2))=3$ (see also  \cite[Lemma 2.9]{Mahmoudifar}).

In what follows, we will concentrate on nonabelian groups $G$ in which the centralizer of every noncentral element of $G$ is abelian. Such groups are called AC-groups.  The smallest nonabelian AC-group is $S_3$. As a matter of fact, there are many infinite families of AC-groups, such as: 
\begin{itemize}
\item Dihedral groups $D_{2k}$ $(k\geqslant 3)$, defined by  $$D_{2k}=\langle x, y \ | \ x^k=y^2=1, yxy^{-1}=x^{-1} \rangle.$$ 
\item Semidihedral groups $SD_{2^k}$ $(k\geqslant 4)$,  defined by 
$$SD_{2^k}=\langle x, y \ | \ x^{2^{k-1}}=y^2=1, yxy^{-1}=x^{-1+2^{k-2}}\rangle. $$
\item Generalized quaternion groups $Q_{4k}$ $(k\geqslant 2)$,  defined by
$$Q_{4k}=\langle x, y \ | \ x^{2k}=1, y^2=x^k,  yxy^{-1}=x^{-1}\rangle. $$
\item Simple groups $L_2(2^k)$ $(k\geqslant 2)$,  and general linear groups ${\rm GL}(2, q)$, $q=p^k>2$, $p$  a prime. 
\end{itemize}

We now return to the general case. Let $G$ be a nonabelian $AC$-group of order $n$ with center $Z=Z(G)$ of order $m$. 
Then,  by Eq. (2), we have 
$${\cal C}(G)=K_m\vee \Delta(G),$$
where $\Delta (G)={\cal C}(G\setminus Z)$.  It is easy to see that $\Lambda:=\{C_G(x) \  | \ x\in G\setminus Z\}$ is the set of all
maximal abelian subgroups of $G$. Let $t:=|\Lambda|$ and $$\Lambda=\{C_G(x_1), C_G(x_2), \ldots, C_G(x_t)\}.$$ 
Put $C_i:=C_G(x_i)\setminus Z$ and $m_i=|C_i|$, for $i=1, 2, \ldots, t$. Then, we get
$$\Delta(G)=\bigoplus_{i=1}^{t} {\cal C}(C_i)=\bigoplus_{i=1}^{t} K_{m_i}.$$
It follows by Lemma \ref{elementary0} (1) that the eigenvalues of Laplacian matrix $L_{\Delta(G)}$ are:
$$0, \underbrace{m_1, m_1,  \ldots, m_1}_{m_1-1}, \  0, \underbrace{m_2, m_2,  \ldots, m_2}_{m_2-1}, \ldots,   \ 0, \underbrace{m_t, m_t, \ldots, m_t}_{m_t-1}. $$
Therefore, using Lemma \ref{elementary0} (2),  the eigenvalues of Laplacian matrix $L_{{\cal C} (G)}$ are:
$$n, \ \underbrace{n,  \ldots, n}_{m-1}, \  \underbrace{m_1+m,  \ldots, m_1+m}_{m_1-1}, \ldots,   \  \underbrace{m_t+m,  \ldots, m_t+m}_{m_t-1},  \ \underbrace{m,  \ldots, m}_{t-1}, \ 0, $$
and  using Eq. (\ref{eq1}), we get  the following theorem (see also \cite[Corollary 2.5.]{Mahmoudifar}):
\begin{theorem}\label{th-AC}  Let $G$ be a finite nonabelian $AC$-group of order $n$ with center of order $m$. 
Let $C_G(x_1), C_G(x_2), \ldots, C_G(x_t)$ be all distinct centralizers of noncentral elements of $G$ and
$m_i=|C_G(x_i)\setminus Z(G)|$, for $i=1, 2, \ldots, t$. Then, there holds
$$\kappa(G)=n^{m-1}m^{t-1}\prod_{i=1}^{t} (m_i+m)^{m_i-1}.$$
In particular, if $G$ is a centerless AC-group, then  we have 
$$\kappa(G)=\prod_{i=1}^{t} (m_i+1)^{m_i-1}=\prod_{i=1}^{t} |C_G(x_i)|^{|C_G(x_i)|-2}. $$
\end{theorem}
Theorem  \ref{th-AC}, together with some rather technical computations  (see proofs of Corollary 3.7 and Propositions 4.1--4.3 in \cite{Das}) yields
 some special results which are summarized in Table 1.
\begin{center}
{\bf Table 1.}  $\kappa (G)$ for some  special AC-groups $G$.
\end{center}
$$
\begin{array}{lclllll}
\hline 
G &     & n &  m & t  &  m_i & \kappa(G) \\
\hline 
D_{2k} &  k \ \mbox{odd} & 2k & 1 & k+1 & k-1, 1, \ldots, 1 & k^{k-2} \\[0.1cm]
D_{2k} &  k \ \mbox{even} & 2k & 2 & \frac{k}{2}+1 & k-2, 2, \ldots, 2 & 2^{\frac{3k+2}{2}}k^{k-2} \\[0.1cm]
Q_{4k} &  k\geqslant 2  & 4k & 2 & k+1 & 2k-2, 2, \ldots, 2 & 2^{5k-1}k^{2k-2} \\[0.1cm]
SD_{2^k} &  k\geqslant 4  & 2^k  & 2 & 2^{k-1}+1 & 2^{k-1}-2, 2, \ldots, 2 & 2^{(2^{k-2}-1)(2k+1)+4}\\[0.1cm]
P &  p \ \mbox{prime} & p^3& p & p+1 & p^2-p, \ldots, p^2-p & p^{2p^3-5}\\ 
\hline
\end{array}
$$

\vspace{0.3cm}

Similarly, if $G$ is one of the almost simple groups $L_2(2^k)$ or ${\rm GL}(2,q)$, where 
$k\geqslant 2$ and $q=p^f>2$ ($p$ is a prime), then using technical computations similar to those in the proofs of Propositions 4.4 and 4.5
in \cite{Das},  we obtain:\\

$G=L_2(2^k)$:  \ \ $n=2^k(2^{2k}-1)$, $m=1$, $t=2^{4k-2}+2^k+1$, and $m_i$:
$$ \underbrace{2^k-1, \ldots, 2^k-1}_{2^k+1}, \  \underbrace{2^k-2, \ldots, 2^k-2}_{2^{k-1}(2^k+1)}, \ 
\underbrace{2^k, \ldots, 2^k}_{2^{k-1}(2^k-1)}. $$

$G={\rm GL}(2, q)$:  \ \  $n=(q^2-1)(q^2-q)$, $m=q-1$, $t=q^2+q+1$, and $m_i$:
$$ \underbrace{q^2-3q+2, \ldots, q^2-3q+2}_{q(q+1)/2},   \underbrace{q^2-q, \ldots, q^2-q}_{q(q-1)/2}, 
\underbrace{q^2-2q+1, \ldots, q^2-2q+1}_{q+1}.$$
Therefore, by Theorem \ref{th-AC}, we get 
$$\kappa(L_2(2^k))=2^{k(2^k-2)(2^k+1)} (2^k-1)^{2^{k-1}(2^k-3)(2^k+1)} (2^k+1)^{2^{k-1}(2^k-1)^2},$$
(compare this with Theorem \ref{char2}), and 
$$\kappa({\rm GL}(2, q))=q^{q^3-q^2-q-2} (q-1)^{\frac{q(3q^3+5)}{2}-2q^3-2q^2-4} (q+1)^{\frac{q(q^3+3)}{2}-q^3-2}.$$

In the next result,  we deal with an AC-group $G$ for which $G/Z(G)\cong \mathbb{Z}_p\times \mathbb{Z}_p$. 

\begin{theorem}\label{pp-abelian}  Let $G$ be a  group of order $n$  with nontrivial center $Z(G)$ of order $m$ such that 
$G/Z(G)\cong {\Bbb Z}_p\times {\Bbb Z}_p$,  where $p$ is  a prime. Then, there holds
$$\kappa(G)=p^{n+m-p-3}m^{n-2}.$$
\end{theorem}
\begin{proof}  First, we claim that $G$ is an AC-group. To show this, suppose  $x$ is a noncentral element of $G$. Then $Z(G)<C_G(x)<G$, and since $|G:Z(G)|=p^2$, 
we conclude that $|C_G(x):Z(G)|=p$.  This shows that $C_G(x)=\langle x, Z(G)\rangle$, which is an abelian group, and so $G$ is an AC-group, as claimed.
 
Since $G/Z(G)$ is an elementary abelian $p$-group of order $p^2$,  it has exactly $p+1$ distinct subgroups
of order $p$, say $\langle Z(G)x_1\rangle, \langle Z(G)x_2\rangle, \ldots, \langle Z(G)x_{p+1}\rangle$.
We claim that $C_G(x_1), C_G(x_2), \ldots, C_G(x_{p+1})$ are all distinct centralizers of noncentral elements 
of $G$. Suppose $y$ is an arbitrary noncentral element of $G$. Then $Z(G)y$ as an element of $G/Z(G)$ lies in 
a unique subgroup of $G/Z(G)$ of order $p$, say $\langle Z(G)x_i\rangle$. Hence, $y=zx_i$ for some $z\in Z(G)$.  This yields that $C_G(y)=C_G(x_i)$, as claimed.  

Therefore, in the notation of Theorem \ref{th-AC},  we have $n=p^2m$, $t=p+1$ and $m_i=(p-1)m$ for $i=1, 2, \ldots, p+1$. Now a direct computation shows that 
$$\kappa(G)=p^{n+m-p-3}m^{n-2},$$ as required.
\end{proof}

\begin{corollary}\label{2-abelian-trees}  Let $G$ be a group having a $2$-abelian 
partition and let $Z(G)$ be its center of order $m$. Then $$\kappa(G)=2^{5m-5}m^{4m-2}.$$ \end{corollary}
\begin{proof}  Using Theorem \ref{th-1},  $G/Z(G)\cong \mathbb{Z}_2\times \mathbb{Z}_2$. Now,  it follows by Theorem \ref{pp-abelian} that
$\kappa(G)=2^{5m-5}m^{4m-2}$, as required. \end{proof}
\begin{corollary}\label{3-abelian-trees}  Let $G$ be a group having a $3$-abelian 
partition and let $Z(G)$ be its center of order $m$.  The following conditions hold: 
\begin{itemize}
\item[\rm (a)] $G/Z(G)\cong {\Bbb Z}_2\times {\Bbb Z}_2$ and $\kappa(G)=2^{5m-5}m^{4m-2}$.
\item[\rm (b)] $G/Z(G)\cong {\Bbb Z}_3\times {\Bbb Z}_3$ and $\kappa(G)=3^{10m-6}m^{9m-2}$.
\item[\rm (c)] $G/Z(G)\cong {\Bbb S}_3$ and $\kappa(G)=2^{4m-4}3^{3m-2}m^{6m-1}$.
\end{itemize}
\end{corollary}
\begin{proof}  Using Theorem \ref{th-2},  we treat separately the different cases:
\begin{itemize}
\item[\rm (1)] $G/Z(G)\cong {\Bbb Z}_p\times {\Bbb Z}_p$ where $p\in \{2, 3\}$.   
In this case, it follows from Theorem \ref{pp-abelian} that 
$$\kappa(G)=p^{m(p^2+1)-p-3}m^{mp^2-2},$$
whence also (a) and (b) hold.
\item[\rm (2)] $G/Z(G)\cong {\Bbb S}_3$.  Using Theorem 5 in \cite{Belcastro}, $G$ has only five 
distinct centralizers. 
In fact, in the notation of Theorem  \ref{th-2}, the distinct centralizers of $G$, written as 
unions of right cosets of $Z=Z(G)$, are 
\begin{itemize}
\item[] $C_1=Z\cup Zx\cup Zx^2=\langle Z, x\rangle$,
\item[]  $C_2=Z\cup Zy=\langle Z, y\rangle$,
\item[] $C_3=Z\cup Zyx=\langle Z, xy\rangle$,
\item[] $C_4=Z\cup Zyx^2=\langle Z, xy^2\rangle$.
\end{itemize}
This shows that $G$ is an AC-group, and thus using Theorem \ref{th-AC}, we obtain 
$\kappa(G)=2^{4m-4}3^{3m-2}m^{6m-1},$ and (c) follows.
\end{itemize}
The proof is complete.
\end{proof}
\begin{center}
 {\bf Acknowledgments }
\end{center}
This work was done during the second author had a visiting position at the
Department of Mathematical Sciences, Kent State
University, USA. He would like to thank the hospitality of the Department of Mathematical Sciences of KSU. This
research was also supported by Iran National Science Foundation (INSF:
94028936) while the first author was a Postdoctoral Research
Associate at the Faculty of Mathematics of K. N. Toosi University
of Technology.

\noindent {\sc  A. Mahmoudifar}\\[0.2cm]
{\sc Faculty of Mathematics, K. N. Toosi
University of Technology,
 P. O. Box $16315$--$1618$, Tehran, Iran,}\\[0.1cm]
 {\em E-mail address}: {\tt alimahmoudifar@gmail.com}\\[0.3cm]
 {\sc A. R. Moghaddamfar}\\[0.2cm]
{\sc Faculty of Mathematics, K. N. Toosi
University of Technology,
 P. O. Box $16315$--$1618$, Tehran, Iran,}\\[0.1cm]
 {\sc and}\\[0.1cm]
 {\sc Department of Mathematical Sciences, Kent State
University,}\\ {\sc  Kent, Ohio $44242$, United States of
America}\\[0.1cm]
{\em E-mail addresses:}:  {\tt
moghadam@kntu.ac.ir}, and {\tt amoghadd@kent.edu}\\[0.3cm]
 {\sc  F. Salehzadeh}\\[0.2cm]
{\sc Faculty of Mathematics, K. N. Toosi
University of Technology,
 P. O. Box $16315$--$1618$, Tehran, Iran,}\\[0.1cm]
 {\em E-mail address}: {\tt salehzadeh.fayez@gmail.com}\\[0.3cm]

\begin{thebibliography}{99}
\bibitem{Abdollahi} 
A. Abdollahi, S. M. Jafarian Amiri and A. M.  Hassanabadi,  Groups with specific number of centralizers, {\em Houston J. Math.},  33(1) (2007), 43--57.

\bibitem{Belcastro} S. M. Belcastro and G. J. Sherman, Counting centralizers in finite groups, 
{\em Math. Mag.}, 5 (1994), 111--114.

\bibitem{Bertram} E. A. Bertram, Some applications of graph theory to finite groups, {\em Discrete Math.,} 44 (1)(1983), 31--43. 

\bibitem{Biggs} N. Biggs,   {\em Algebraic Graph Theory},  Cambridge University Press, London, 1974.

\bibitem{Brandstadt}  A. Brandst${\rm \ddot{a}}$dt,  Partitions of graphs into one or two independent sets and cliques, {\em Discrete Math.}, 152(1-3) (1996), 47--54.

\bibitem{Britnell} J. R. Britnell and N. Gill, Perfect commuting graphs, {\em J. Group Theory}, 20(1)(2017), 71-–102.

\bibitem{Das} A. K. Das and D. Nongsiang, 
On the genus of the commuting graphs of finite non-abelian groups,
{\em Int. Electron. J. Algebra}, 19 (2016), 91--109. 

\bibitem{Foldes-Hammer} S. F$\rm \ddot{o}$ldes and P. L. Hammer, Split graphs,
{\em Proceedings of the Eighth Southeastern Conference on
Combinatorics, Graph Theory and Computing (Louisiana State Univ.,
Baton Rouge, La., $1977$)}, 311--–315.

\bibitem{Hup} B. Huppert, {\em Endliche Gruppen I}, Springer, Berlin, 1967.

\bibitem{Ito} N. It${\rm \hat{o}}$,  On finite groups with given conjugate types I,
{\em Nagoya Math. J.},  6 (1953), 17--28.

\bibitem{Mahmoudifar} A. Mahmoudifar and A. R. Moghaddamfar, Commuting graphs of groups and related numerical parameters, {\em Comm. Algebra},  45(7)(2017), 3159--3165.

\bibitem{Merris} R. Merris, Laplacian graph eigenvectors, {\em Linear Algebra Appl.},  278 (1998),  221--236.

\bibitem{Pyber} L. Pyber, The number of pairwise noncommuting elements and the index of the centre in a finite group, {\em J. London Math. Soc.},  35(2) (1987), 287--295.

\bibitem{Scorza}   G. Scorza, I gruppi che possono pensarsi come somme di tre loro sottogruppi, {\em Boll. Un Mat. Ital.}, (1926) 216--218.
\end{thebibliography}
\end{document}